\documentclass[11pt]{article}
\usepackage{times}
\usepackage{amssymb,amsmath}
\usepackage{epsf}
\usepackage{times,bm,mathrsfs}  
\usepackage{amsmath}
\usepackage{amssymb}
\usepackage{amsfonts}
\usepackage{mathrsfs}
\usepackage{bbm}
\usepackage{color}
\usepackage{graphicx}
\usepackage{amscd}
\usepackage{epsfig}

\definecolor{Red}{rgb}{1,0,0}
\definecolor{Blue}{rgb}{0,0,1}
\definecolor{Olive}{rgb}{0.41,0.55,0.13}
\definecolor{Green}{rgb}{0,1,0}
\definecolor{MGreen}{rgb}{0,0.8,0}
\definecolor{DGreen}{rgb}{0,0.55,0}
\definecolor{Yellow}{rgb}{1,1,0}
\definecolor{Cyan}{rgb}{0,1,1}
\definecolor{Magenta}{rgb}{1,0,1}
\definecolor{Orange}{rgb}{1,.5,0}
\definecolor{Violet}{rgb}{.5,0,.5}
\definecolor{Purple}{rgb}{.75,0,.25}
\definecolor{Brown}{rgb}{.75,.5,.25}
\definecolor{Grey}{rgb}{.5,.5,.5}
\definecolor{Black}{rgb}{0,0,0}

\def\path{{\tt path}}

\newcommand{\bdm}{\begin{displaymath}}
\newcommand{\edm}{\end{displaymath}}
\newcommand{\bea}{\begin{eqnarray*}}
\newcommand{\eea}{\end{eqnarray*}}
\newcommand{\bean}{\begin{eqnarray}}
\newcommand{\eean}{\end{eqnarray}}

\newcommand{\var}{\mathrm{Var}}

\newtheorem{theorem}{Theorem}

\newtheorem{corollary}{Corollary}
\newtheorem{definition}{Definition}

\newtheorem{lemma}{Lemma}

\newenvironment{proof}{\noindent{\textbf{Proof:}}}{$\blacksquare$\vskip\belowdisplayskip}

\title{Reconstruction of Random Colourings}

\author{Allan Sly \thanks{{Email: sly@stat.berkeley.edu Dept. of Statistics, U.C. Berkeley. Supported by NSF grants DMS-0528488 and DMS-0548249 }}}

\begin{document}
\maketitle

\begin{abstract}
Reconstruction problems have been studied in a number of contexts
including biology, information theory and and statistical physics.  We consider the reconstruction problem for random $k$-colourings on
the $\Delta$-ary tree for large $k$.  Bhatnagar et. al.
\cite{BhVeVi:07} showed non-reconstruction when $\Delta \leq \frac12
k\log k - o(k\log k)$.  We tighten this result and show non-reconstruction
when $\Delta \leq k[\log k + \log \log k + 1 - \ln 2 -o(1)]$ which is very close to the known upper bound on the number of colours needed for  reconstruction, $\Delta \geq k[\log k + \log \log k + 1+o(1)]$.

\end{abstract}

\section{Introduction}

Determining the reconstruction threshold of a Markov random field has been of interest in a number of areas including biology, information theory and and statistical physics.  Reconstruction thresholds on trees are believed to determine the dynamic phase transitions in many constraint satisfaction problems including random K-SAT and random colourings on random graphs.  It is thought that at this poin the space of solutions splits into exponentially many clusters.  The properties of the space of solutions of these problems are of interest to  physicists, probabilists and theoretical computer scientists.

The upper bound on the number of colours required for reconstruction of $ k[\log k + \log \log k + 1+o(1)]\geq \Delta$ was known from the work of \cite{MoPe:03} and \cite{Semerjian:08} and is given by the analysis of a naive reconstruction algorithm which reconstructs the root only when it is known with absolute certainty given the leaves. The problem of finding good lower bounds on the number of colours required for reconstruction is more difficult, it requires showing that the spins on the root and the leaves are asymptotically independent.  The best previous rigorous results for a lower bound on the number of colours needed for reconstruction was $\Delta \leq \frac12 k\log k - o(k\log k)$ in \cite{BhVeVi:07}.  We improve this to $\Delta \leq k[\log k + \log \log k + 1 - \ln 2 -o(1)]$.  Even at a heuristic level no lower bound as good as ours was known.

\subsection{Definitions}
We begin by giving a general description of broadcast models on
trees and the reconstruction problem.  The broadcast model on a tree
$T$ is a model in which information is sent from the root $\rho$
across the edges, which act as noisy channels, to the leaves of $T$.
For some given finite set of characters $\mathcal{C}$ a
configuration on $T$ is an element of $\mathcal{C}^T$, that is an
assignment of a character $\mathcal{C}$ to each vertex. The
broadcast model is a probability distribution on configurations
defined as follows.  Some $|\mathcal{C}|\times |\mathcal{C}|$
probability transition matrix $M$ is chosen as the noisy channel on
each edge.  The spin $\sigma_\rho$ is chosen from $\mathcal{C}$
according to some initial distribution and is then is propagated
along the edges of the tree according to the transition matrix $M$.
That is if vertex $u$ is the parent of $v$ in the tree then spin at
$v$ is defined according to the probabilities
\[
P(\sigma_v = j|\sigma_u=i)=M_{i,j}.
\]
We will focus on the colouring model with $|\mathcal{C}|=k$ which is
given by the transition matrix $M_{i,j} = \frac{1_{i\neq j}}{k-1}$.

Broadcast models and in particular colourings can also be considered
as Gibbs measures on trees. Given a finite set of colours $k$ and a
graph $T=(V,E)$ a $k$-colouring is an assignment of a colour to each
vertex so that adjacent vertices have different colours. The random
$k$-coluring model is then the uniform probability distribution on
valid $k$-colourings of the graph.  It is a Gibbs measure or Markov
random field on the space of configurations $\sigma\in
\{1,\ldots,k\}^V$ given by
\[
P(\sigma)=\frac1{Z}\prod_{(u,v)\in E} 1_{\sigma_u\neq \sigma_v}
\]
where $Z$ is a normalizing constant given by the number of
colourings of $T$.  On an infinite tree more than one Gibbs measure
may exist, the broadcast colouring model corresponds to the free
Gibbs measure.

We will restrict our attention to $\Delta$-ary trees, that is the
infinite rooted tree where every vertex has $\Delta$ offspring.  Let
$L(n)$ denote the spins at distance $n$ from the root and let
$L^i(n)$ denote $L(n)$ conditioned on $\sigma_\rho=i$.
\begin{definition}
We say that a model is \emph{reconstructible} on a tree $T$ if for
some $i,j\in \mathcal{C}$,
\[
\limsup_n d_{TV} (L^i(n),L^j(n)) > 0
\]
where $d_{TV}$ is the total variation distance. When the limsup is 0
we will say the model has \emph{non-reconstruction} on $T$.
\end{definition}
Non-reconstruction is equivalent to the mutual information between
$\sigma_\rho=L(0)$ and $L(n)$ going to 0 as $n$ goes to infinity and
also to $\{L(n)\}_{n=1}^\infty$ having a non-trivial tail
sigma-field. More equivalent formulations are given in
\cite{Mossel:04} Proposition 2.1.

In contrast consider the uniqueness property of a model.
\begin{definition}
We say that a model has \emph{uniqueness} on a tree $T$ if
\[
\limsup_n \quad\sup_{L,L'} \ \ \ d_{TV}
(P(\sigma_\rho=\cdot|L(n)=L),P(\sigma_\rho=\cdot|L(n)=L')) > 0
\]
where the supremum is over all configurations $L,L'$ on the vertices
at distance $n$ from the root.
\end{definition}
Reconstruction implies non-uniqueness and is a strictly stronger
condition.  Essentially uniqueness says that there is some
configuration on the leaves which provides information on the root
while reconstruction says that a typical configuration on the leaves
provides information on the root.

\subsection{Background}
For some parameterized collection of models the key question in
studying reconstruction is finding which models have reconstruction,
which typically involves finding a threshold. This problem naturally
arises in biology, information theory and statistical physics and
involves the trade off between increasing numbers of leaves with
increasingly noisy information as the distance from the root to the
leaves increases. The simplest collection of model is the binary
symmetric channel which is defined on two characters with
\[
M=\left(\begin{array}{cc}1-\epsilon &\epsilon\\ \epsilon &
1-\epsilon\end{array}\right)
\]
for $0<\epsilon<\frac12$ which corresponds to the ferromagnetic
Ising model on the tree with no external field.  It was shown in
\cite{BlRuZa95} and \cite{EvKePeSch:00} that this channel has
reconstruction if and only if $\Delta(1-2\epsilon)^2>1$.

The broadcast model is a natural model for the evolution of
characters of DNA.  In phylogenetic reconstruction the goal is to
reconstruct the ancestry tree of a collection of species given their
genetic data. Daskalakis, Mossel and Roch
\cite{Mossel:04b,DaMoRo:06} proved the conjecture of Mike Steel that
the number of samples required for phylogenetic reconstruction
undergoes a phase transition at the reconstruction threshold for the
binary symmetric channel.

Exact reconstruction thresholds have only been calculated in the
binary symmetric model and binary asymmetric models with
sufficiently small asymmetry \cite{BoChMoRo:06}.  In both these
cases the threshold corresponds to the Kesten-Stigam bound
\cite{KesSti:06}. The Kesten-Stigam bound shows that reconstruction
holds whenever $\Delta\lambda_2(M)^2 > 1$ where $\lambda_2(M)$
denotes the second largest eigenvalue of $M$.  In fact when
$\Delta\lambda_2(M)^2 > 1$ it is possible to asymptotically
reconstruct the root from just knowing the number of times each
character appears on the leaves (census reconstruction) without
using the information on their positions on the leaves. Mossel
\cite{Mossel:01,Mossel:04} showed that the Kesten-Stigam bound is
not the bound for reconstruction in the binary-asymmetric model with
sufficiently large asymmetry or in the Potts model with sufficiently
many characters.

It was shown by \cite{Jonasson:02} that $k$-colourings have
uniqueness on $\Delta$-ary trees if and only if $k\geq \Delta+2$
which is a lower bound on the reconstruction threshold.  Exactly
finding the threshold for reconstruction is difficult so most
attention has been focused on finding its asymptotics as the number
of colours and the degree goes to in infinity. Recently
\cite{BhVeVi:07} greatly improved this bound showing that
non-reconstruction holds when $k\geq C\Delta/\log{\Delta}$ (or
equivalently when $\Delta \leq C^{-1} k\log k$) for $C>2$.  On the
other hand \cite{MoPe:03} showed that when $\Delta \geq k\log k +
o(k\log k)$ then with high probability in $k$ the spin of the root
is exactly determined by the leaves and so reconstruction is
possible.  With a more detailed analysis this argument can be improved to show reconstruction when $k[\log k + \log \log k + 1+o(1)]\geq \Delta$, as was shown in \cite{Semerjian:08}.

This is a large improvement on the Kesten-Stigam bound
which implies reconstruction when $\Delta > (k-1)^2$.  In related
work Mezard and Montanari \cite{MezMon:06} found a variational
principle which establishes bounds on reconstruction for colourings.  However, except in the case of $k=5$ this does not improve upon the bound from Lemma \ref{l:fixedRoot}, see \cite{Semerjian:08,ZdKr:07}.
Our results  establish far tighter bounds with the upper and lower
bounds differing by just $k\log 2$ rather than $\frac12 k \log k$.

\begin{theorem}\label{t:main}
The $k$-colouring model on the $\Delta$-ary tree has reconstruction
when
\[
\Delta \geq k[\log k + \log \log k + 1+o(1)].
\]
and  non-reconstruction when
\[
\Delta \leq k[\log k + \log \log k + 1 - \ln(2) -o(1)].
\]
\end{theorem}

\subsection{Applications to  Statistical Physics}

The reconstruction threshold on trees is believed to play a critical
role in the dynamic phase transitions in certain glassy systems given by random constraint satisfaction problems such as random K-SAT and random colourings on random graphs. We will briefly describe what is conjectured by
physicists about such systems \cite{KMRSZ:07,ZdKr:07},
generally without rigorous proof, and why understanding the
reconstruction threshold for colourings plays an important role in
such systems.

The Erd\H{o}s-R\'enyi graph $G(n,p)$ is a random graph on $n$
vertices where every pair of vertices is connected  with probability
$p$.  To maintain constant average degree $\Delta$ we let
$p=\Delta/n$. The $k$-colouring model on $G(n,\Delta/n)$ or random
$\Delta$-regular graphs undergoes several phase transitions as
$\Delta$ grows. If we consider the space of solutions to a random
colouring model where two colourings are adjacent if they differ at
a single vertex then for the smallest values of $\Delta$ the space
of solutions forms a large connected component.  Above the
clustering transition $\Delta_d$ the space of solutions breaks into
exponentially many disconnected clusters and has no giant component
with a constant fraction of the probability. This replica symmetry
breaking transition is believed to occur at $\Delta_d = k[\log k +
\log \log k + \alpha + o(1)]$ with differing values of $\alpha$
conjectured: $\alpha=1-\ln 2$ in \cite{KrPaWe:04} and $\alpha=1$ in \cite{KMRSZ:07}.
In a recent remarkable result \cite{AchlioptasCoOg:08} rigorously proved that when $(1+o(1)) k\log k \leq \Delta \leq (2- o(1)) k\log k$ then the space of solutions indeed breaks into exponentially many small clusters. A second
transition occurs when most clusters have frozen spins, that is
vertices which have the same colour in every colouring in the
cluster. This phase transition is believed to occur at $\Delta_r =
k[\log k + \log \log k + 1 + o(1)]$ \cite{Semerjian:08,ZdKr:07} and is the best upper bound known for $\Delta_d$. Two more
transitions are believed to occur, condensation where the size of
the clusters takes on a Poisson-Dirichlet process and the colouring
threshold beyond which no more colourings are possible, are
conjectured to occur at $\Delta_c = 2k\log k - \log k - 2\log 2
+o(1)$ and $\Delta_s= 2k\log k - \log k - 1 +o(1)$ respectively
\cite{ZdKr:07}. Similar results are also expected to hold for K-SAT
and other random constraint satisfiability problems \cite{KMRSZ:07}.

Both random regular and Erd\H{o}s-R\'enyi random graphs are locally
tree-like.  Asymptotically in a random regular graph the
neigbourhood of a random vertex is a random regular tree  and for
Erd\H{o}s-R\'enyi random graphs it is a Galton-Watson branching
process tree with Poisson offspring distribution. It is conjectured
\cite{KMRSZ:07} that the reconstruction threshold on the
corresponding tree is exactly the clustering threshold $\Delta_d$ on
the random graph. As such rigorous estimates of the reconstruction
problem can be seen as part of a larger program of understanding
glassy systems from constraint satisfaction problems.

The clustering threshold is also believed to play an important role
in the efficiency of MCMC algorithms for finding and sampling from
colourings of the graphs.  MCMC algorithms are believed to be
efficient up to the clustering threshold but experience an
exponential slowdown beyond it \cite{KMRSZ:07}.  This is to be
expected since a local MCMC algorithm can not move between clusters
each of which has exponentially small probability.  Rigorous proofs
of rapid mixing of MCMC algorithms, such as the Glauber dynamics,
fall a long way behind.  For random regular graphs, results of
\cite{DFHV:04} imply rapid mixing when $k\geq 1.49\Delta$, well
below the reconstruction threshold and even the uniqueness
threshold. Even less is known for Erd\H{o}s-R\'enyi random graphs as
almost all MCMC results are given in terms of the maximum degree
which in this case grows with $n$. Polynomial time mixing of the
Glauber dynamics has been shown \cite{MoSly:07} for a constant
number of colours in terms of $\Delta$.

\subsection{Open Problem}

If the probability that the leaves uniquely determine the spin at
the root does not go to 0 as $n$ goes to infinity then the model has
reconstruction.  It is natural to ask is this a necessary condition for reconstruction.  When $k=5$ and $\Delta=14$ it was shown in \cite{MezMon:06} using a variational principle that reconstruction holds but the probability that the leaves fix the root goes to 0.  However, this is the only case in which the variational principle gives an upper bound on the number of colours required for reconstruction which is better than the bound of the leaves fixing the root.  It remains open to determine if for large numbers of colours/degrees if this is exactly the
reconstruction threshold. Numerical results of \cite{ZdKr:07} suggest this
is in fact not the case and there are two separate thresholds.
Answering this question would be of significant interest.

\section{Proofs}

We introduce the notation we use in the proofs.  We denote the
colours by $\mathcal{C}=\{1,\ldots,k\}$ and let $T$ be the
$\Delta$-ary tree rooted at $\rho$. Let $u_1,\ldots,u_\Delta$ be the
children of $\rho$ and let $T_j$ denote the subtree of descendants
of $u_j$. Let $P(\sigma)$ denote the free measure on colourings on
the $\Delta$-ary tree.  Let $L(n)$ denote the spins at distance $n$
from $\rho$ and let $L_j(n)$ denote the spins on level $n$ in the
subtree $T_j$. Let $L^{i}(n)$ and $L^{i}_j(n)$ respectively denote
$L(n)$ and $L_j(n)$ conditioned on $\sigma_\rho=i$.  For a boundary condition $L$ on the spins at distance $n$ from $\rho$ define the deterministic function $f_n$ as
\[
f_n(i,L) = P(\sigma_\rho=i | L(n) = L).
\]
By the recursive nature of the tree we also have that
\[
f_n(i,L) = P(\sigma_{u_j}=i | L_j(n) = L).
\]
Now define  $X_i(n)=X_i$ by
\[
X_i(n)=f_n(i,L(n)).
\]
These random variables are a deterministic function of the random
configuration $L(n)$ of the leaves which gives the marginal
probability that the root is in state $i$.  By symmetry the $X_i$
are exchangable. Now we define two distributions
\[
X^+ = X^{+}(n)=f_n(1,L^{1}(n))
\]
and
\[
X^- = X^{-}(n)=f_n(1,L^{2}(n))
\]
which are different distributions because $L^1(n)$ and $L^2(n)$ are
different distributions of configurations on the leaves. We will
establish non-reconstruction by showing that $X^+$ and $X^-$ both
converge to $\frac1k$ in probability as $n$ goes to infinity.  Let
$x_n$ and $z_n$ denote $EX^{+}(n)$ and $E(X^{+}(n)-\frac1k)^2$
respectively.   By symmetry we have
\[
f_n(i_2,L^{i_1}(n))\stackrel{d}{=}\begin{cases} X^+ & i_1=i_2, \\
X^- & \hbox{otherwise,} \end{cases}
\]
and the set $\{f_n(i,L^{1}(n)):2\leq i \leq k\}$ is
exchangeable.  Moreover they are conditionally exchangeable given $f_n(1,L^{1}(n))$.  Now define
\[
Y_{ij}= Y_{ij}(n) = f_n(i,L^{1}_j(n)).
\]
This is equal to the probability that $\sigma_{u_j}=i$ given the random configuration $L^{1}_j(n)$ on the spins on level $n$ in the subtree $T_j$.  As
the configurations on the subtrees are conditionally independent given the spin at
the root we have that the random vectors $Y_j=\{Y_{1j},\ldots,
Y_{kj} \}$ are conditionally independent.  Further given $\sigma_j$ and $Y_{\sigma_j j}$ the random variables $\{Y_{ij}: i\neq \sigma_j \}$ are conditionally exchangeable.  We make use of these symmetries to simplify the anaylsis. Given the standard Gibbs measure
recursions on trees we have that
\[
f_{n+1}(1,L(n+1))= \frac{\prod_{j=1}^\Delta  (1-
f_n(1,L_j(n)))}{\sum_{i=1}^k \prod_{j=1}^\Delta  (1-
f_n(i,L_j(n)))}
\]
and so
\[
X^{+}(n+1)= \frac{Z_1}{\sum_{i=1}^k Z_i}
\]
where
\[
Z_i = \prod_{j=1}^\Delta (1-Y_{ij}).
\]

The following lemma, which can be viewed as the analogue of Lemma 1 of
\cite{BCRM:06}, allows us to relate the first and second moments of
$X^+$.

\begin{lemma}\label{l:changeOfMeasure}
We have that
\[
x_n = EX^+= E \sum_{i=1}^k P(\sigma_\rho=i|L^{1}(n))^2 = E
\sum_{i=1}^k (X_i(n))^2.
\]
and
\[
x_n-\frac1k = EX^+-\frac1k=  E \sum_{i=1}^k (X_i(n)-\frac1k)^2\geq
E(X^+-\frac1k)^2 = z_n.
\]
\end{lemma}
\begin{proof}
From the definition of conditional probabilities and of $f_n$ and the fact that
$P(\sigma_\rho=1)=\frac1k$ we have that
\begin{align*}
Ef_n(1,L^{1}(n)) &= \sum_L f_n(1,L)
P(L(n) = L|\sigma_\rho=1)\\
&= \sum_L \frac{P(L(n)=L,\sigma_\rho=1)}{P(\sigma_\rho=1)} f_n(1,L)\\
&= k\sum_L P(L(n)=L) f_n(1,L)^2\\
&= k E(X_1(n))^2\\
&= E \sum_{i=1}^k (X_i(n))^2
\end{align*}
and
\[
E \sum_{i=1}^k (X_i(n)-\frac1k)^2  = E \sum_{i=1}^k (X_i(n))^2 -
\frac2k E \sum_{i=1}^k X_i(n) + k\frac1{k^2}= EX^+-\frac1k.
\]
\end{proof}

\begin{corollary}\label{c:changeOfMeasure}
We have that $x_n \geq\frac1k$ and that
\[
\lim_n x_n = \frac1k.
\]
implies non-reconstruction.
\end{corollary}
\begin{proof}
We have that $x_n \geq z_n+\frac1k \geq \frac1k$.  If $x_n$
converges to $\frac1k$ then
\[
\sum_{i=1}^k E\left(X_i(n)-\frac1k\right)^2\rightarrow 0
\]
which implies non-reconstruction.
\end{proof}

\subsection{Non-reconstruction}

Our analysis is split into two phases, the first when $x_n$ is close
to 1 and the second when $x_n$ is close to $\frac1k$.

\begin{lemma}\label{l:DecayBoundsNear1}
Suppose that $\beta<1-\log 2$.  Then for sufficiently large $k$ if
$\Delta<k[\log k + \log \log k +\beta]$ then
\[
\limsup_n x_n \leq \frac2k.
\]
\end{lemma}

\begin{proof}
We fix the colour of the root to be 1 let $\mathcal{F}$ denote the
sigma-algebra generated by $\{\sigma_{u_j}:1\leq j \leq \Delta\}$
the colours of the neighbours of the root.  For $2\leq i \leq k$
let $b_i=\#\{j:\sigma_{u_j}=i\}$, the number of times each colour
appears amongst the neighbours of the root.  For $1\leq i \leq k$
define
\[
U_i=\prod_{1\leq j \leq \Delta:\sigma_{u_j}= i} (1-Y_{ij}).
\]
We will use the symmetries and exchangeability of the model to reduce the problem to considering a random variable only involving the $U_i$.  Conditional on $\mathcal{F}$, the $U_i$ are independent and are distributed as the
product of $b_i$ independent copies of $X^+(n)$ and $0\leq U_i \leq 1$ for all $i$. Fix an $\ell$ with $2\leq \ell \leq k$.  Let $W_1$ and $W_\ell$ be defined by
\[
W_1=  \prod_{1\leq j \leq
\Delta:\sigma_{u_j}\neq \ell} (1-Y_{1j}),\quad W_\ell = \prod_{1\leq j \leq
\Delta:\sigma_{u_j}\neq \ell} (1-Y_{\ell j})
\]
so $Z_\ell =W_\ell  U_\ell$. Define
\[
\widetilde{Z}_\ell =W_1 U_\ell ,
\]
and
\[
\widetilde{Z}_1 = W_\ell \prod_{1\leq j \leq
\Delta:\sigma_{u_j}= \ell}(1-Y_{1j}),
\]
and for $i\not\in \{1,\ell\}$,
\[
\widetilde{Z}_i = Z_i.
\]

As we noted earlier $Y_j=\{Y_{1j},\ldots,
Y_{kj} \}$ are conditionally independent given $\mathcal{F}$ and for each $j$ given $\sigma_j$ and $Y_{\sigma_j j}$ the random variables $\{Y_{ij}: i\neq \sigma_j \}$ are conditionally exchangeable. It follows that
\begin{align}\label{e:exchangeability}
& \left(W_1,W_\ell,Z_1,\ldots,Z_k,U_1\ldots,U_k, \sigma_1,\ldots,\sigma_\Delta \right)\nonumber\\
 \stackrel{d}{=} & \left(W_\ell,W_1, \widetilde{Z}_1,\ldots, \widetilde{Z}_k,U_1\ldots,U_k, \sigma_1,\ldots,\sigma_\Delta \right).
\end{align}
where we denote equality as in distributions of random vectors.  Since $(W_1 + \sum_{i=2}^n Z_i) - (W_\ell + \sum_{i=2}^n \widetilde{Z}_i) = (W_1-W_\ell)(1-U_\ell)$ has the same sign as $W_1-W_\ell$ then $ \frac{1}{W_1 + \sum_{i=2}^n Z_i} - \frac{1}{W_\ell + \sum_{i=2}^n \widetilde{Z}_i} $ has the opposite sign as  $W_1-W_\ell$.
Applying the equality in distribution of equation \eqref{e:exchangeability} we have that
\begin{align*}
&E\left[ \left . \frac{W_\ell - W_1}{W_1 + \sum_{i=2}^n Z_i}\right|  \mathcal{F},\{U_i\} \right]\\
=& \frac12 E\left[ \left . \frac{W_\ell - W_1}{W_1 + \sum_{i=2}^n Z_i} + \frac{W_1 - W_\ell}{W_\ell + \sum_{i=2}^n \widetilde{Z}_i}\right|  \mathcal{F},\{U_i\} \right]\\
=&  \frac12 E\left[ \left . \left( W_\ell - W_1 \right) \left( \frac{1}{W_1 + \sum_{i=2}^n Z_i} - \frac{1}{W_\ell + \sum_{i=2}^n \widetilde{Z}_i} \right) \right|  \mathcal{F},\{U_i\} \right]\\
 \geq  & \ \ 0
\end{align*}
where the first equality follows using equality in distributions of the random vectors and the inequality follows from the two terms of the product having the same sign.  Since $0\leq Z_1\leq W_1\leq 1$ we have that,
\begin{align*}
E\left[ \left . \frac{Z_1}{Z_1 + \sum_{i=2}^n Z_i}\right|  \mathcal{F},\{U_i\} \right]&\leq E\left[ \left . \frac{W_1}{W_1 + \sum_{i=2}^n Z_i}\right|  \mathcal{F},\{U_i\} \right]\\
&\leq  E\left[ \left . \frac{W_\ell}{W_1 + \sum_{i=2}^n Z_i}\right|  \mathcal{F},\{U_i\} \right]\\
& \leq E\left[ \left . \frac{W_\ell}{Z_1 + \sum_{i=2}^n Z_i}\right|  \mathcal{F},\{U_i\} \right]
\end{align*}
and so
\[
E\left[ \left . \frac{Z_1\left(1+\sum_{i=2}^k U_i\right)}{ \sum_{i=1}^n Z_i}\right|  \mathcal{F},\{U_i\} \right] \leq \sum_{i=1}^k E\left[ \left . \frac{Z_i}{\sum_{i=1}^n Z_i}\right|  \mathcal{F},\{U_i\} \right]=1
\]
and hence
\[
E\left[ \left . X^+(n+1) \right|  \mathcal{F},\{U_i\} \right] = E\left[ \left . \frac{Z_1}{ \sum_{i=1}^n Z_i}\right|  \mathcal{F},\{U_i\} \right] \leq \frac1{1+\sum_{i=2}^k U_i}.
\]

Now using the fact that $\frac1{1+x}=\int_0^1 s^x ds$ we have that
\[
\frac{1 }{1+ \sum_{i=2}^{k} U_i} = \int_0^1 s^{\sum_{i=2}^{k}
U_i} ds
\]
As $s^u$ is convex as a function of $u$ we have that $s^u\leq
s^0(1-u) + s^1 u$ when $0\leq u \leq 1$ and so since $0\leq U_i \leq 1$ we
have that $E s^{U_i} \leq  (1-E U_i) + s EU_i = 1 -(1-s)EU_i$.   Since conditional on $\mathcal{F}$ the $U_i$ are independent and are distributed as the product of $b_i$ independent copies of $X^+(n)$ we have that,
\begin{align*}
E \left[\left. X^{+}(n+1)\right|\mathcal{F}\right] &\leq \int_0^1
\prod_{i=2}^{k} (1 -(1-s)E[U_i|\mathcal{F}]) ds\\ &= \int_0^1
\prod_{i=2}^{k} (1 -(1-s)(1-x_n)^{b_i}) ds.
\end{align*}
Now the colours $\sigma_{u_j}$ are chosen independently and
uniformly from the set $\{2,\ldots,k\}$ so $(b_2,\ldots,b_k)$ has a
multinominal distribution. Let $\beta<\beta^*< 1-\log 2$ and let
$\widetilde{b}_i$ be iid random variables distributed as
Poisson$(D)$ where $D=\log k + \log \log k +\beta^*$.  By Lemma
\ref{l:multinomial} we can couple the $b$'s and $\widetilde{b}$'s so
that $(b_2,\ldots,b_k) \leq
(\widetilde{b}_2,\ldots,\widetilde{b}_k)$ whenever $\sum_{j=2}^k
\widetilde{b}_j \geq  \Delta$.  It follows that
\begin{align*}
x_{n+1} &=E X^{+}(n+1)\\
& \leq E1_{\{\sum_{j=2}^k \widetilde{b}_j <\Delta\}} + \int_0^1
E\prod_{i=2}^{k} (1 -(1-s)(1-x_n)^{\widetilde{b}_i}) ds\\
& \leq p + \int_0^1  \left(1 -(1-s)\exp(-x_n D)\right)^{k-1} ds\\
& \leq p + \int_0^1  \exp\left(-(1-s)(k-1)\exp(-x_n D)\right) ds\\
& = p + \frac{1 - \exp\left(-(k-1)\exp(-x_n D)\right) }{(k-1)\exp(-x_n D)}\\
\end{align*}
where $p=P(\hbox{Poisson}((k-1)D)<\Delta)$.  Now
$p=\exp(-\Omega(\frac{k}{\sqrt{\Delta}}))=o(k^{-1})$ and the
function
\[
g(y) = p + \frac{1 - \exp\left(-(k-1)\exp(-y D)\right)
}{(k-1)\exp(-y D)}
\]
is increasing in $y$ so the result follows by Lemma
\ref{l:sequenceDecay}.

\end{proof}

\begin{lemma}\label{l:sequenceDecay}
Let $y_0,y_1,\ldots$ be a sequence of positive real numbers such
that $y_0=1$ and $y_{n+1}=g(y_n)$ where $g(y_n)=p + \frac{1 -
\exp\left(-(k-1)\exp(-y_n D)\right) }{(k-1)\exp(-y_n D)}$, $D=\log k
+ \log \log k +\beta^*$, $\beta^*< 1-\log 2$ and $p=o(k^{-1})$. Then
for large enough $k$,
\[
\limsup_n x_n < \frac2{k}.
\]
\end{lemma}

\begin{proof}
Since $\left . \frac{d}{dx} \frac{1-e^{-x}}{x}\right |_{x=0}=
-\frac12$ we can find $\epsilon,\delta>0$ such that when
$0<x<\delta$, then
\[
\frac{1-e^{-x}}{x}<1-\left(\frac12-\epsilon\right)x.
\]
We also choose $r'>r>0$ such that $(\frac12-\epsilon)
e^{-\beta^*}>e^{-1}(1+r')$.   Now for large enough $k$,
$(k-1)\exp(-D)=\frac{(k-1)e^{-\beta^*}}{k \log k}<\delta$ and so
\[
y_1=g(1)\leq p+
1-\left(\frac12-\epsilon\right)\frac{(k-1)e^{-\beta^*}}{k \log
k}\leq 1- p+ \frac{(1+r)e^{-1}}{\log k}\leq 1-\frac{e^{-1}}{\log k}.
\]
Now since $g$ is a continuous increasing function and $y_1<y_0$ it
follows that the sequence $y_i$ is decreasing.  Suppose that
$(k-1)\exp(-y_i D) < \delta$.  Then
\[
y_{i+1} \leq p + 1 - \left(\frac12-\epsilon\right)(k-1)\exp(-y_i D)
\]
and so
\begin{align*}
1-y_{i+1} & \geq  \left(\frac12-\epsilon\right)(k-1)\exp(-y_i D) - p\\
& \geq \left(\frac12-\epsilon\right)\frac{(k-1)e^{-\beta^*} }{k \log k}\exp((1-y_i)\log k) - p\\
& \geq \frac{(1+r')e^{-1} }{\log k}\exp((1-y_i)\log k) - p\\
& \geq (1+r')(1-y_i) - p\\
& \geq (1+r)(1-y_i)
\end{align*}
where the second last inequality uses the fact that $e^x\geq e x$
and the final inequality uses the fact that $1-y_i\geq
\frac{e^{-1}}{\log k}$ while $p=o(k^{-1})$.  It follows that $y_i$
decreases until for some $i$, $(k-1)\exp(-y_i D) \geq \delta$.  Now
let $\frac{1 - e^{-\delta}}{\delta}=\alpha'<\alpha''<\alpha<1$ for
some $\alpha$.  When $k$ is large enough then
\[
y_{i+1} \leq p + \frac{1 - e^{-\delta}}{\delta} \leq \alpha''.
\]
Then again for $k$ large enough, $\exp(-y_{i+1} D) \geq
\exp(-\alpha''D)\geq \exp(-\alpha \log k)=k^{-\alpha}$.  It follows
that
\[
y_{i+2} \leq p + \frac1{(k-1)\exp(-y_{i+1} D)} \leq 2 k^{\alpha-1}.
\]
Finally we have $\exp(-y_{i+2} D) \geq \exp(-2k^{\alpha-1} D)\geq
\frac23$ and so
\[
y_{i+3}  \leq p + \frac1{(k-1)\exp(-y_{i+2} D)} < 2 k^{-1}
\]
when $k$ is large enough which completes the proof.

\end{proof}

In the preceding lemma we note that the requirement that
$\beta^*<1-\ln 2$ comes from the fact that $x<\frac12 e^{x-\beta^*}$
for all $x$ when $\beta^*<1-\ln 2$.

\begin{lemma}\label{l:multinomial}
Suppose that $(b_1,\ldots,b_k)$ has the multinominal distribution
$M(n,(\frac1k,\frac1k,\ldots\frac1k))$. Let $\widetilde{b}_j$ be iid
random variables distributed as Poisson$(D)$.  We can couple the
$b$'s and $\widetilde{b}$'s so that $(b_1,\ldots,b_k) \leq
(\widetilde{b}_1,\ldots,\widetilde{b}_k)$ (respectively $\geq$)
whenever $\sum_{j=1}^k \widetilde{b}_j \geq n$ (respectively
$\leq$).
\end{lemma}

\begin{proof}
Since the $\widetilde{b}_j$ are independent and Poisson, conditional
on the sum $N=\sum_{j=1}^k \widetilde{b}_j$, the distribution of
$(\widetilde{b}_1,\ldots,\widetilde{b}_k)$ is multinominal
$M(N,(\frac1k,\frac1k,\ldots\frac1k))$ (see \cite{Lange:03}
Proposition 6.2.1). Now if $n\leq m$ then two multinomial
distributions $A$ and $B$ distributed as
$M(n,(\frac1k,\frac1k,\ldots\frac1k))$ and
$M(m,(\frac1k,\frac1k,\ldots\frac1k))$ respectively can be trivially
coupled so that $A\leq B$ which completes the proof.
\end{proof}

Janson and Mossel \cite{JanMos:04} studied ``robust
reconstruction'', the question of when reconstruction is possible
from a very noisy copy of the leaves.  They found that the threshold
for robust reconstruction is exactly the Kesten-Stigam bound. Lemma
\ref{l:DecayBoundsNear1} establishes that the leaves provide very
little information about the spin at a vertex a long distance from
the leaves. So as information over long distances is very noisy the
results of \cite{JanMos:04} suggest that reconstruction would only
be possible after the Kesten-Stigam bound whereas, in our context,
$\Delta$ is much less than $\lambda_2(M)^{-2}$.  As such only crude
bounds are needed to establish the following lemma.

\begin{lemma}\label{l:contraction}
For sufficiently large $k$ if $\Delta \leq 2k\log k$ and if $x_n
\leq \frac2k$ then
\[
x_{n+1}-\frac1k \leq \frac12\left(x_n-\frac1k\right).
\]
\end{lemma}

\begin{proof}
Using the identity
\[
\frac1{s+r}=\frac1s - \frac{r}{s^2} + \frac{r^2}{s^2}\frac1{s+r}
\]
and taking $s=E \sum_{i=1}^k Z_i$ and $r=\sum_{i=1}^k (Z_i-E Z_i)$
we have that
\begin{align*}
x_{n+1}-\frac1k &= E \frac{Z_1-\frac1k\sum_{i=1}^k Z_i}{\sum_{i=1}^k
Z_i}\\
&= E \frac{Z_1-\frac1k\sum_{i=1}^k Z_i}{E \sum_{i=1}^k Z_i} -E
\frac{\left(Z_1-\frac1k\sum_{i=1}^k Z_i\right)\left(\sum_{i=1}^k
(Z_i-E Z_i)\right)}{\left(E \sum_{i=1}^k Z_i\right)^2}\\
& \quad + \frac{Z_1-\frac1k\sum_{i=1}^k Z_i}{\sum_{i=1}^k
Z_i}\frac{\left(\sum_{i=1}^k (Z_i-E Z_i)\right)^2}{\left(E
\sum_{i=1}^k Z_i\right)^2}.
\end{align*}
Now by Lemma \ref{l:ZEstimates},
\begin{align}\label{e:contractionA}
\frac{E\left(Z_1-\frac1k\sum_{i=1}^k Z_i\right)}{E \sum_{i=1}^k Z_i}
& \leq
\frac{\frac{k-1}{k}\left(1+\frac{2\Delta}k\left(x_n-\frac1k\right)\right)
-\frac{k-1}{k}\left(1-\frac{2\Delta}{k^2}\left(x_n-\frac1k\right)\right)}
{1+(k-1)\left(1-\frac{2\Delta}{k^2}\left(x_n-\frac1k\right)\right)}\nonumber \\
& \leq \frac{3\Delta}{k^2} \left(x_n-\frac1k\right).
\end{align}
Using the inequality $\frac12(a^2+b^2) \geq ab$ we have that
\begin{align*}
&-\left(Z_1-\frac1k\sum_{i=1}^k Z_i\right)\left(\sum_{i=1}^k Z_i-E
Z_i\right)\\
 &= -\left(\left(Z_1-E Z_1\right) + \left(E Z_1 -
\frac1k\sum_{i=1}^k E
Z_i\right) - \frac1k\left(\sum_{i=1}^k (Z_i - E Z_i)\right)\right)\\
&\quad \cdot\left(\sum_{i=1}^k (Z_i-E Z_i)\right)\\
& \leq \frac12\left|Z_1-E Z_1\right|^2 + \left(\frac12 +
\frac1k\right)\left|\sum_{i=1}^k (Z_i-E Z_i) \right|^2\nonumber\\
&\quad - \left(EZ_1 - E\frac1k\sum_{i=1}^k
Z_i\right)\left(\sum_{i=1}^k (Z_i-E Z_i)\right)
\end{align*}
so by Lemma \ref{l:ZEstimates} we have that,
\begin{eqnarray*}
E\left[-\left(Z_1-\frac1k\sum_{i=1}^k Z_i\right)\left(\sum_{i=1}^k
(Z_i-E Z_i)\right)\right]\\
\leq
\left(\frac{k-1}{k}\right)^{2\Delta}\left(x_n-\frac1k\right)
\left[\frac{4\Delta}k + 4\Delta \right]
\end{eqnarray*}
and
\begin{align}\label{e:contractionB}
&E \left[-\frac{\left(Z_1-\frac1k\sum_{i=1}^k
Z_i\right)\left(\sum_{i=1}^k (Z_i-E Z_i)\right)}{\left(E
\sum_{i=1}^k Z_i\right)^2}\right]\nonumber\\
 \leq &
\frac{\left(x_n-\frac1k\right) \left[\frac{4\Delta}k + 4\Delta
\right]}{\left(
1+(k-1)\left(1-\frac{2\Delta}{k^2}\left(x_n-\frac1k\right)\right)\right)^2}\nonumber\\
\leq & \frac{5\Delta}{k^2} \left(x_n-\frac1k\right).
\end{align}
Finally since $0\leq \frac{Z_1}{\sum_{i=1}^k Z_i}\leq 1$ we have
that $\left| \frac{Z_1 - \frac1k \sum_{i=1}^k Z_i}{\sum_{i=1}^k Z_i}
\right|\leq 1$ and so
\begin{align}\label{e:contractionC}
E \frac{Z_1-\frac1k\sum_{i=1}^k Z_i}{\sum_{i=1}^k
Z_i}\frac{\left(\sum_{i=1}^k (Z_i-E Z_i)\right)^2}{\left(E
\sum_{i=1}^k Z_i\right)^2} & \leq E \frac{\left(\sum_{i=1}^k (Z_i-E
Z_i)\right)^2}{\left(E \sum_{i=1}^k Z_i\right)^2}\nonumber \\
& \leq
\frac{5\Delta}{k^2} \left(x_n-\frac1k\right).
\end{align}
Combining equations \eqref{e:contractionA}, \eqref{e:contractionB}
and \eqref{e:contractionC} we have that
\begin{equation}\label{e:contractionKS}
x_{n+1}-\frac1k \leq \frac{13\Delta}{k^2}
\left(x_n-\frac1k\right)\leq \frac12 \left(x_n-\frac1k\right)
\end{equation}
and for large enough $k$, which completes the result.

\end{proof}

\begin{lemma}\label{l:ZEstimates}
For sufficiently large $k$ if $\Delta \leq 2k\log k$ and if $x_n
\leq \frac2k$ then the following all hold
\begin{equation}\label{e:Z1Mean}
\left(\frac{k-1}{k}\right)^\Delta \leq EZ_1  \leq
\left(\frac{k-1}{k}\right)^\Delta\left(1+\frac{2\Delta}k\left(x_n-\frac1k\right)\right)
\end{equation}
and for $i\neq 1$,
\begin{equation}\label{e:ZiMean}
\left(\frac{k-1}{k}\right)^\Delta\left(1-\frac{2\Delta}{k^2}\left(x_n-\frac1k\right)\right)
\leq EZ_i  \leq \left(\frac{k-1}{k}\right)^\Delta,
\end{equation}

\begin{equation}\label{e:Z1Var}
\var Z_1  \leq
\left(\frac{k-1}{k}\right)^{2\Delta}\frac{4\Delta}k\left(x_n-\frac1k\right),
\end{equation}

\begin{equation}\label{e:ZSumVar}
\var \left(\sum_{i=1}^k Z_i \right) \leq
\left(\frac{k-1}{k}\right)^{2\Delta}
4\Delta\left(x_n-\frac1k\right).
\end{equation}

\end{lemma}

\begin{proof}

From equation \eqref{e:AppendixZ1} we have that
\[
E Z_1= \left(\frac{k-1}k +
\frac1{k-1}\left(x_n-\frac1k\right)\right)^\Delta
\]
and so since $x_n\geq \frac1k$ by Corollary \ref{c:changeOfMeasure}
we have that
\[
E Z_1 \geq \left(\frac{k-1}k \right)^\Delta.
\]
Then since $\exp(x)=1+x + O(x^2)$ and
$\frac{k\Delta}{(k-1)^2}\left(x_n-\frac1k\right)$ is small for large
$k$,
\begin{align*}
E Z_1 &\leq  \left(\frac{k-1}k \right)^\Delta
\exp\left(\frac{k\Delta}{(k-1)^2}\left(x_n-\frac1k\right)\right)\\
&\leq
\left(\frac{k-1}{k}\right)^\Delta\left(1+\frac{2\Delta}k\left(x_n-\frac1k\right)\right)
\end{align*}
which establishes equation \eqref{e:Z1Mean}.  Equations
\eqref{e:ZiMean},\eqref{e:Z1Var} and \eqref{e:ZSumVar} are
established similarly.

\end{proof}



\subsection{Reconstruction}
An upper bound on the number of colours needed for reconstruction is found by estimating the probability that the colour of the root is
uniquely determined by the colours at the leaves.  This method was described in \cite{MoPe:03} and used to a higher level of precision in \cite{Semerjian:08}.  We restate the result and give a full proof for completeness.
\begin{lemma}\label{l:fixedRoot}
Suppose that $\beta>1$.  Then for sufficiently large $k$ if $\Delta
> k[\log k + \log \log k +\beta]$ then the colour of the root is
uniquely determined by the colours at the leaves with probability at
least $1- \frac1{\log k}$, that is
\[
\inf_n P(X^+(n) = 1) > 1 - \frac1{\log k}.
\]

\end{lemma}

\begin{proof}
Let $p_n$ be the probability that the leaves at distance $n$
determine the spin at the root, that is $p_n=P(X^+(n)=1)$. We will
show that when $k$ is large then $\liminf_n p_n$ is close to 1.

Suppose we fix the colour of the root to be 1 let $\mathcal{F}$
denote the sigma-algebra generated by $\{\sigma_{u_j}:1\leq j \leq
\Delta\}$ the colours of the the neighbours of the root.  For $2\leq
i \leq k$ let $b_i=\#\{j:\sigma_{u_j}=i\}$, the number of times each
colour appears in the neighbours of the root.  Now each colour
$\sigma_{u_j}$ is chosen uniformly from the set $\{2,\ldots,k\}$ so
$(b_2,\ldots,b_k)$ has a multinominal distribution. Let
$\beta>\beta^*> 1$ and let $\widetilde{b}_i$ be iid random variables
distributed as Poisson$(D)$ where $D=\log k + \log \log k +\beta^*$.
By Lemma \ref{l:multinomial} we can couple the $b$'s and
$\widetilde{b}$'s so that $(b_2,\ldots,b_k) \geq
(\widetilde{b}_2,\ldots,\widetilde{b}_k)$ whenever $\sum_{i=2}^k
\widetilde{b}_j \leq \sum_{i=2}^k b_j = \Delta$.  If for each colour
$2\leq i \leq k$ there is some vertex $u_j$ such that the leaves in
$L^1_j(n)$ fix the colour of $u_j$ to be $i$ then the leaves
$L^1(n+1)$ fix the colour of $\rho$ to be 1. Conditional on
$\mathcal{F}$ the probability that there is such a vertex $u_j$ for
a given colour $i$ is at least $1-(1-p_n)^{b_i}$. Moreover these are
conditionally independent of $\mathcal{F}$ so it follows that
\begin{align*}
p_{n+1} &\geq \prod_{i=2}^k E\left[ 1 -
(1-p_n)^{b_i}|\mathcal{F}\right]\\
& \geq \prod_{i=2}^k E\left[ 1 - (1-p_n)^{\widetilde{b_i}}\right] -
s\\
& = (1-\exp(- p_n D))^{k-1} -
s\\
\end{align*}
where $s=P(\hbox{Poisson}((k-1)D)>\Delta)=o(k^{-1})$.  Now
$f(x)=(1-\exp(- p_n D))^{k-1} - s$ is increasing in $x$ and when $k$
is large enough
\begin{align*}
f\left(1-\frac1{\log k}\right) &= \left(1-\exp\left(- (1-\frac1{\log
k}) (\log k + \log \log k +\beta^*)\right)\right)^{k-1} - s\\
& > 1-\frac1{\log k}
\end{align*}
and since $p_0=1$,
\[
\inf_n p_n \geq 1-\frac1{\log k}
\]
which completes the proof.

\end{proof}

\subsection{Main Theorem}
\begin{proof}(Theorem \ref{t:main})
Combining Lemmas \ref{l:DecayBoundsNear1} and \ref{l:contraction}
establishes non-reconstruction when $\Delta \leq k[\log k + \log
\log k + 1 - \ln(2) -o(1)]$.  Lemma \ref{l:fixedRoot} shows that the
root can be reconstructed correctly with probability at least
$1-\frac1{\log k}$ which establishes reconstruction when $\Delta
\geq k[\log k + \log \log k + 1 -o(1)]$.
\end{proof}

\subsection*{Remarks}
For large $k$ the Poisson$(\Delta)$ distribution is concentrated
around $\Delta$ with standard deviation $O(\sqrt{\Delta})$ which is
significantly smaller than the errors bounds in Theorem
\ref{t:main}.  With some minor modifications the bounds for
$\Delta$-ary trees can be extended to Galton-Watson branching
processes with offspring distribution Poisson$(\Delta)$.  The
reconstruction of Galton-Watson branching processes with offspring
distribution Poisson$(\Delta)$ is of interest because, as noted
before, it is believed to be related to the clustering transitions
for colourings on Erd\H{o}s-R\'enyi random graphs.

To be more specific for the proof of non-reconstruction we can again
bound $x_n = EX^+$ where the expected value is taken over all
possible trees. In Lemma \ref{l:DecayBoundsNear1} we repeat the same
bounds on $x_n$, the only difference being $\Delta$ is now random,
which does not affect the results for large $k$.  Then similar
estimates can be made in Lemma \ref{l:contraction} provided
$\frac{\Delta}{k} \left(x_n-\frac1k\right)$ is very small.  As
$\Delta$ is concentrated around it's expected value the probability
of this not holding is very small and this can be used to complete
the proof of non-reconstruction.

When $\beta>\beta^*>1$, with probability going to 1 as $k$ goes to
infinity, the Galton-Watson branching process contains a subgraph
which is a $(k[\log k + \log \log k +\beta^*])$-ary tree rooted at
$\rho$. Reconstruction then follows from Lemma \ref{l:fixedRoot}.

\subsection*{Acknowledgments}
The author would like to thank Elchanan Mossel for his useful
comments and advice and thank Dror Weitz, Nayantara Bhatnagar, Lenka  Zdeborova, Florent Krz{\c a}ka{\l}a, Guilhem Semerjian and Dmitry Panchenko for useful discussions.

\appendix
\section{Appendix}

Observe that since $EX^+ + (k-1)EX^- =1$ we have that $E X^+
-\frac1k = -(k-1) (E X^- -\frac1k)$. We will show that the means and
variances of the $Y_{ij}$ and $Z_{i}$ can all be calculated in terms
of $x_n$ and $z_n$.

\begin{lemma}\label{l:AppendixY}
We have the identities
\begin{equation}\label{e:AppendixY1}
E Y_{1j} = \frac1k - \frac1{k-1}(x_n-\frac1k)
\end{equation}

\begin{equation}\label{e:AppendixY2}
E Y_{1j}^2 = \frac1{k^2} + \frac{k-2}{k(k-1)}
\left(x_n-\frac1k\right) - \frac1{k-1}z_n.
\end{equation}
For $2\leq i \leq k$,
\begin{equation}\label{e:AppendixY3}
E Y_{ij} = \frac1k + \frac1{(k-1)^2}\left(x_n-\frac1k\right)
\end{equation}
and
\begin{equation}\label{e:AppendixY4}
E Y_{ij}^2 = \frac1{k^2} + \frac{k^2 -2k+2}{k(k-1)^2}
\left(x_n-\frac1k\right) + \frac1{(k-1)^2}z_n.
\end{equation}
For any $1 \leq i_1 < i_2 \leq k$,
\begin{equation}\label{e:AppendixY5}
\hbox{Cov}(Y_{i_1 j},Y_{i_2 j})\leq 0
\end{equation}
\end{lemma}

\begin{proof}
When the root is conditioned to be 1, $\sigma_j\neq 1$ and so
$Y_{1j}$ is distributed as $X^-$ and we have that,
\[
E Y_{1j} = EX^-=\frac1k - \frac1{k-1}[EX^+-\frac1k]=\frac1k -
\frac1{k-1}(x_n-\frac1k)
\]
and
\begin{align*}
E Y_{1j}^2 &= E (X^-)^2\\
&= \frac1{k-1}[E \sum_{i=1}^k (X_i)^2 - E (X^+)^2]\\
&=\frac1{k-1}[E X^+ - E (X^+)^2]\\
&= \frac1{k-1}[\frac{k-2}{k} E (X^+ -\frac1k) - E (X^+ - \frac1k)^2 + \frac{k-1}{k^2}]\\
&= \frac1{k^2} + \frac{k-2}{k(k-1)} \left(x_n-\frac1k\right) -
\frac1{k-1}z_n.
\end{align*}
where the third equality follows from Lemma \ref{l:changeOfMeasure}.
For $2\leq i \leq k$ we have that
\[
E Y_{ij} = \frac1{k-1}[1-EY_{1j}]=\frac1{k-1}[1-\frac1k +
\frac1{k-1}[x_n-\frac1k]]=\frac1k +
\frac1{(k-1)^2}\left(x_n-\frac1k\right)
\]
and again using Lemma \ref{l:changeOfMeasure}
\begin{align*}
E Y_{ij}^2 &= \frac1{k-1}[E \sum_{i=1}^k (X_i)^2 - E Y_{1j}^2]\\
&= \frac1{k^2} + \frac{k^2 -2k+2}{k(k-1)^2} \left(x_n-\frac1k\right)
+ \frac1{(k-1)^2}z_n
\end{align*}
Also for $2\leq i \leq k$,
\begin{align*}
EY_{1j}Y_{ij}&=\frac1{k-1}\sum_{i'=2}^k
EY_{1j}Y_{i'j}\\&=\frac1{k-1} EY_{1j}(1-Y_{1j})\\
&\leq \frac1{k-1} EY_{1j}E(1-Y_{1j})\\
&=EY_{1j}EY_{ij}
\end{align*}
so $\hbox{Cov}(Y_{1j},Y_{ij})\leq 0$.  Finally for $2 \leq i_1<
i_2\leq k$
\[
\hbox{Var}(1-Y_{1j})=\sum_{i=2}^k
\hbox{Var}(Y_{ij})+(k-1)(k-2)\hbox{Cov}(Y_{i_1 j},Y_{i_2 j})
\]
and so
\begin{align*}
\hbox{Cov}(Y_{i_1 j},Y_{i_2 j})&=\hbox{Var}(1-Y_{1j})-\sum_{i=2}^k
\hbox{Var}(Y_{ij})\\
&\leq \hbox{Var}(X^-)-((k-2)\hbox{Var}(X^-) + \hbox{Var}(X^+))\\
&\leq 0
\end{align*}
so $\hbox{Cov}(Y_{i_1 j},Y_{i_2 j})\leq 0$.

\end{proof}

Using Lemma \ref{l:AppendixY} we can calculate the means and
covariances of the $Z_j$.

\begin{lemma}\label{l:AppendixZ}
We have the following results
\begin{equation}\label{e:AppendixZ1}
E Z_1= (\frac{k-1}k + \frac1{k-1}\left(x_n-\frac1k\right))^\Delta
\end{equation}

\begin{equation}\label{e:AppendixZ2}
E Z_1^2= \left(\left(\frac{k-1}k\right)^2 +
\frac{3k-2}{k(k-1)}\left(x_n-\frac1k\right) - \frac1{k-1}z_n
\right)^\Delta
\end{equation}
For each $2\leq i \leq k$ then
\begin{equation}\label{e:AppendixZ3}
E Z_i = (\frac{k-1}k -
\frac1{(k-1)^2}\left(x_n-\frac1k\right)))^\Delta
\end{equation}
and
\begin{equation}\label{e:AppendixZ4}
E Z_i^2 = \left(\left(\frac{k-1}k\right)^2 +
\frac{k^2-4k+2}{k(k-1)^2}\left(x_n-\frac1k\right) +
\frac1{(k-1)^2}z_n\right)^\Delta
\end{equation}
For any $1 \leq i_1 < i_2 \leq k$,
\begin{equation}\label{e:AppendixZ5}
\hbox{Cov}(Z_{i_1 j},Z_{i_2 j})\leq 0
\end{equation}
\end{lemma}

\begin{proof}
By equation \eqref{e:AppendixY1} we have that
\begin{align*}
E Z_1 &= E\prod_{j=1}^\Delta (1-Y_{1j})\\
&= \left(1-\left(\frac1k - \frac1{k-1}\left(x_n-\frac1k\right)\right)\right)^\Delta\\
&= \left(\frac{k-1}k +
\frac1{k-1}\left(x_n-\frac1k\right)\right)^\Delta.
\end{align*}
which establish equation \eqref{e:AppendixZ1}.  Equations
\eqref{e:AppendixZ2}, \eqref{e:AppendixZ3} and \eqref{e:AppendixZ4}
follow similarly. Using equation \eqref{e:AppendixY5} we have that
for $1 \leq i_1 < i_2 \leq k$,
\begin{align*}
E Z_{i_1} Z_{i_2} &= E\prod_{j=1}^\Delta (1-Y_{i_1 j})(1-Y_{i_2 j})\\
&\leq \prod_{j=1}^\Delta E(1-Y_{i_1 j})E(1-Y_{i_2 j})\\
&= E Z_{i_1} EZ_{i_2}
\end{align*}
which establishes equation \eqref{e:AppendixZ5}.
\end{proof}

\bibliographystyle{plain}
\bibliography{allbib}

\begin{thebibliography}{10}

\bibitem{AchlioptasCoOg:08}
Dimitris Achlioptas and Amin Coja-Oghlan.
\newblock Algorithmic barriers from phase transition.
\newblock http://front.math.ucdavis.edu/0803.2122, 2008.

\bibitem{BhVeVi:07}
Nayantara Bhatnagar, Juan Vera, and Eric Vigoda.
\newblock Reconstruction for colorings on trees.
\newblock http://front.math.ucdavis.edu/0711.3664, 2007.

\bibitem{BlRuZa95}
Ruiz~J. Bleher, P.~M. and Zagrebnov~V. A.
\newblock On the purity of limiting gibbs state for the ising model on the
  bethe lattice.
\newblock {\em J. Stat. Phys}, 79:473--–482, 1995.

\bibitem{BoChMoRo:06}
Christian Borgs, Jennifer Chayes, Elchanan Mossel, and Sebastien Roch.
\newblock The kesten-stigum reconstruction bound is tight for roughly symmetric
  binary channels.
\newblock In {\em FOCS '06: Proceedings of the 47th Annual IEEE Symposium on
  Foundations of Computer Science (FOCS'06)}, pages 518--530, Washington, DC,
  USA, 2006. IEEE Computer Society.

\bibitem{BCRM:06}
Christian Borgs, Jennifer~T. Chayes, Elchanan Mossel, and S\'{e}bastien Roch.
\newblock The kesten-stigum reconstruction bound is tight for roughly symmetric
  binary channels.
\newblock In {\em FOCS}, pages 518--530. IEEE Computer Society, 2006.

\bibitem{DaMoRo:06}
Constantinos Daskalakis, Elchanan Mossel, and S{\'e}bastien Roch.
\newblock Optimal phylogenetic reconstruction.
\newblock In {\em STOC'06: Proceedings of the 38th Annual ACM Symposium on
  Theory of Computing}, pages 159--168, New York, 2006. ACM.

\bibitem{DFHV:04}
Martin Dyer, Alan Frieze, Thomas~P. Hayes, and Eric Vigoda.
\newblock Randomly coloring constant degree graphs.
\newblock In {\em FOCS '04: Proceedings of the 45th Annual IEEE Symposium on
  Foundations of Computer Science (FOCS'04)}, pages 582--589, Washington, DC,
  USA, 2004. IEEE Computer Society.

\bibitem{EvKePeSch:00}
William Evans, Claire Kenyon, Yuval Peres, and Leonard~J. Schulman.
\newblock Broadcasting on trees and the {I}sing model.
\newblock {\em Ann. Appl. Probab.}, 10(2):410--433, 2000.

\bibitem{JanMos:04}
Svante Janson and Elchanan Mossel.
\newblock Robust reconstruction on trees is determined by the second
  eigenvalue.
\newblock {\em Ann. Probab.}, 32(3B):2630--2649, 2004.

\bibitem{Jonasson:02}
Johan Jonasson.
\newblock Uniqueness of uniform random colorings of regular trees.
\newblock {\em Statist. Probab. Lett.}, 57:243--248, 2002.

\bibitem{KesSti:06}
H.~Kesten and B.~P. Stigum.
\newblock Additional limit theorems for indecomposable multidimensional
  {G}alton-{W}atson processes.
\newblock {\em Ann. Math. Statist.}, 37:1463--1481, 1966.

\bibitem{KMRSZ:07}
Florent Krz{\c a}ka{\l}a, Andrea Montanari, Federico Ricci-Tersenghi, Guilhem
  Semerjian, and Lenka Zdeborova.
\newblock {Gibbs states and the set of solutions of random constraint
  satisfaction problems}.
\newblock {\em Proceedings of the National Academy of Sciences},
  104:10318--10323, 2007.

\bibitem{KrPaWe:04}
Florent Krz{\c a}ka{\l}a, Andrea Pagnani, and Martin Weigt.
\newblock Threshold values, stability analysis, and high-$q$ asymptotics for
  the coloring problem on random graphs.
\newblock {\em Phys. Rev. E}, 70(4):046705, 2004.

\bibitem{Lange:03}
Kenneth Lange.
\newblock {\em Applied probability}.
\newblock Springer Texts in Statistics. Springer-Verlag, New York, 2003.

\bibitem{MezMon:06}
Marc M{\'e}zard and Andrea Montanari.
\newblock Reconstruction on trees and spin glass transition.
\newblock {\em J. Stat. Phys.}, 124(6):1317--1350, 2006.

\bibitem{Mossel:01}
Elchanan Mossel.
\newblock Reconstruction on trees: beating the second eigenvalue.
\newblock {\em Ann. Appl. Probab.}, 11(1):285--300, 2001.

\bibitem{Mossel:04b}
Elchanan Mossel.
\newblock Phase transitions in phylogeny.
\newblock {\em Trans. Amer. Math. Soc.}, 356(6):2379--2404 (electronic), 2004.

\bibitem{Mossel:04}
Elchanan Mossel.
\newblock Survey: information flow on trees.
\newblock In {\em Graphs, morphisms and statistical physics}, volume~63 of {\em
  DIMACS Ser. Discrete Math. Theoret. Comput. Sci.}, pages 155--170. Amer.
  Math. Soc., Providence, RI, 2004.

\bibitem{MoPe:03}
Elchanan Mossel and Yuval Peres.
\newblock Information flow on trees.
\newblock {\em Ann. Appl. Probab.}, 13:817--844, 2003.

\bibitem{MoSly:07}
Elchanan Mossel and Allan Sly.
\newblock Gibbs rapidly samples colorings of g(n,d/n).
\newblock http://arxiv.org/abs/0707.3241, 2007.

\bibitem{Semerjian:08}
Guilhem Semerjian.
\newblock On the freezing of variables in random constraint satisfaction
  problems.
\newblock {\em J.STAT.PHYS.}, 130:251, 2008.

\bibitem{ZdKr:07}
Lenka. Zdeborov{\'a} and Florent. Krz{\c a}ka{\l}a.
\newblock Phase transitions in the coloring of random graphs.
\newblock {\em Phys. Rev. E}, 76:031131, 2007.

\end{thebibliography}

\end{document}